\documentclass[12pt,amstex]{amsart}
\usepackage{etex}
\usepackage{txfonts}
\usepackage{amsmath,amsthm,amsfonts,amssymb,verbatim}
\usepackage{hyperref}
\usepackage{color}

\newtheorem*{thm*}{Theorem}
\newtheorem{thm}{Theorem}[section]

\newtheorem{prop}[thm]{Proposition}
\newtheorem{ques}[thm]{Question}

\theoremstyle{definition}
\newtheorem{defn}[thm]{Definition}
\newtheorem{rem}[thm]{Remark}

\numberwithin{equation}{section}

\def \a {\alpha}
\def \b {\beta}
\def \e {\epsilon}

\def \L {\Lambda}

\def \N {\mathbb{N}}

\def \o {\omega}
\begin{document}
	
\title{Quantitative multiple recurrence for two and three transformations}
\author{Sebasti\'an Donoso}
\address{Centro de Modelamiento Matem\'atico, CNRS-UMI 2807, Universidad de Chile, Beauchef 851, Santiago, Chile.} \email{sdonosof@gmail.com }

\author{Wenbo Sun}
\address{Department of Mathematics, The Ohio State University, 231 West 18th Avenue, Columbus OH, 43210-1174, USA}
\email{sun.1991@osu.edu}


\thanks{The first author is supported by Fondecyt Iniciaci\'on en Investigaci\'on grant 11160061.}

\begin{abstract}
We provide various counter examples for quantitative multiple recurrence problems for systems with more than one transformation. 
We show that 
\begin{itemize}
	\item  There exists an ergodic system $(X,\mathcal{X},\mu,T_1,T_2)$ with two commuting transformations such that for every $0<\ell< 4$, there exists $A\in\mathcal{X}$ such that $$\mu(A\cap T_{1}^{-n}A\cap T_{2}^{-n}A)<\mu(A)^{\ell} \text{ for every } n\neq 0;$$ 
	\item There exists an ergodic system $(X,\mathcal{X},\mu,T_1,T_2, T_{3})$  with three commuting transformations such that for every $\ell>0$, there exists $A\in\mathcal{X}$ such that $$\mu(A\cap T_{1}^{-n}A\cap T_{2}^{-n}A\cap T_{3}^{-n}A)<\mu(A)^{\ell} \text{ for every } n\neq 0;$$
	\item There exists an ergodic system $(X,\mathcal{X},\mu,T_1,T_2)$ with two transformations generating a 2-step nilpotent group such that for every $\ell>0$, there exists $A\in\mathcal{X}$ such that $$\mu(A\cap T_{1}^{-n}A\cap T_{2}^{-n}A)<\mu(A)^{\ell} \text{ for every } n\neq 0.$$   
\end{itemize}
 
\end{abstract}

\maketitle 
\section{Introduction}
\subsection{Quantitative recurrence}
	
The Poincar\'e recurrence theorem states that for every measure preserving system $(X,\mathcal{X},\mu,T)$ and every $A\in\mathcal{X}$ with $\mu(A)>0$, the set $$\{n \in \mathbb{Z}: \mu(A\cap T^{-n}A)>0 \}$$ is infinite. A quantitative version of it was provided by Khintchine \cite{K},
who showed that for every $\epsilon >0$, the set $\{n \in \mathbb{N}: \mu(A\cap T^{-n} A )>\mu(A)^2-\epsilon \}$ is \emph{syndetic}, meaning that it has bounded gaps. 

Multiple recurrence problems refers to the ones concerning the behavior of the set $A\cap T_{1}^{-n} A\cdots \cap T_{d}^{-n}A$. In the case $T_i=T^i$ for an ergodic transformation $T$, Furstenberg \cite{Fu} showed that the set \[\{n \in \mathbb{Z}: \mu(A\cap T^{-n} A\cdots \cap T^{-dn}A) >0 \} \] is infinite. This result is now known as the Furstenberg Multiple Recurrence Theorem. 

The quantitative version of the multiple recurrence problems
concerns not only the positivity of the measure of the set $A\cap T_{1}^{-n} A\cdots \cap T_{d}^{-n}A$, but also how far from 0 this measure is.  More precisely, we ask:
\begin{ques} \label{Q:Poincare} Let $(X,\mathcal{X},\mu,T_{1},\dots,T_{d})$ be a measure preserving system and $F\colon [0,1]\to\mathbb{R}_{\geq 0}$ be a non-negative function. Is the set
\[ \{ n\in \mathbb{Z}: \mu(A\cap T_{1}^{-n} A \cap T_{2}^{-n}A \cap \cdots \cap T_{d}^{-n}A ) \geq F(\mu(A))\}\]
syndetic for all $A\in\mathcal{X}$?
\end{ques}
If the answer to the question is affirmative for some  $T_{1},\dots,T_{d}$ and $F$, we say that $F$ is \emph{good} for $(T_{1},\dots,T_{d})$. Based on the result of Khintchine \cite{K} stating that $F(x)=x^{2}-\e$ is good for $(T)$ (with a single term), a natural conjecture would be that the function $F(x)=x^{d+1}-\e$ is good for $(T_{1},\dots,T_{d})$.
 
 The case $T_{i}=T^{i}$ was solved by
 Bergelson, Host and Kra \cite{BHK}. They
showed that if $(X,\mathcal{X},\mu, T)$ is ergodic, then $F(x)=x^{d+1}-\e$ is good for $(T,T^{2},\dots,T^{d})$ for all $\e>0$ for $d=2$ or 3.
 They also showed two surprising phenomena. First, the hypothesis of ergodicity cannot be removed: 
 there exists a non-ergodic system $(X,\mathcal{X},\mu,T)$ such that $F(x)=x^{\ell}$ is not good for $(T,T^{2})$ for all $\ell>0$.
 Secondly, if $d\geq 4$, then $F(x)=x^{\ell}$ is not good for $(T,T^{2},\dots,T^{d})$ for all $\ell>0$.
  
  In \cite{FK}, Furstenberg and Katznelson proved a multiple recurrence theorem for commuting transformations. For two commuting transformations, its quantitative study was done by Chu \cite{Chu} who proved that for every system $(X,\mathcal{X},\mu,T_1,T_2)$ with two \emph{commuting transformations} $T_1$ and $T_2$ (meaning that $T_{1}T_{2}=T_{2}T_{1}$), ergodic for $\langle T_{1},T_{2}\rangle$, and every $\e>0$, $F(x)=x^{4}-\e$ is good for $(T_{1},T_{2})$.
Here it is worth to stress that the exponent for $(T_{1},T_{2})$ is 4 while the exponent for $(T,T^{2})$ is 3. Indeed, in the same paper, Chu constructed an example showing that $F(x)=x^{3}$ is not good for $(T_{1},T_{2})$. In a later work, Chu and Zorin-Kranich \cite{CZ} improved this example, showing that $F(x)=x^{3.19}$ is not good for $(T_{1},T_{2})$. 
 
In this paper, we study the best component $\ell$ for which $F(x)=x^{\ell}$ is good for $(T_{1},\dots,T_{d})$ in an ergodic system $(X,\mathcal{X},\mu,T_{1},\dots,T_{d})$ with commuting transformations. The result of Chu and Zorin-Kranich \cite{CZ} suggested that the largest $\ell$ not good for $(T_{1},T_{2})$ is between 3.19 and 4. We show that $\ell$ can be sufficiently close to $4$: 

\begin{thm} \label{Thm:2Recurrence}
	There exists a measure preserving system $(X,\mathcal{X},\mu,T_1,T_2)$ with commuting transformations $T_1$ and $T_2$, ergodic for $\langle T_1,T_2 \rangle$ such that for every $0<\ell<4$, $F(x)=x^{\ell}$ is not good for $(T_{1},T_{2})$, i.e. there exists
a set $A\in\mathcal{X}$ such that 
	\[ \mu(A\cap T_1^{-n} A \cap T_2^{-n} A) < \mu(A)^{\ell}   \]
for every  $n\neq 0$.	
\end{thm}

 Chu raised another question \cite{Chu} on whether $F(x)=x^{\ell}$ is good for $(T_{1},T_{2},\dots,T_{d})$ in an ergodic system with $d$ commuting transformations for $d\geq 3$. We show that this is not the case:
\begin{thm} \label{Thm:3Recurrence}
There exists a measure preserving system $(X,\mu,T_1,T_2,T_3)$ with three commuting transformations $T_1$ and $T_2$ and $T_3$, ergodic for $\langle T_1,T_2,T_3 \rangle$ such that for every $\ell>0$, $F(x)=x^{\ell}$ is not good for $(T_{1},T_{2},T_{3})$, i.e.
there exists a subset $A\in\mathcal{X}$ such that 
	
	\[ \mu(A\cap T_1^{-n} A \cap T_2^{-n} A\cap T_3^{-n} A ) < \mu(A)^{\ell}  \]	
 for every  $n\neq 0$.
\end{thm}

If we relax the condition of commutativity of the transformations, the natural condition to look at is nilpotency. Outside the abelian category, we show that there is no polynomial quantitative recurrence even for two transformations $T_1$ and $T_2$ spanning a 2-step nilpotent group. We show 

\begin{thm} \label{Thm:Nilpotent}
	There exists a measure preserving system $(X,\mathcal{X},\mu,T_{1},T_{2})$ such that $T_{1}$ and $T_{2}$ generate a 2-step nilpotent group $\langle T_1,T_2\rangle$, whose actions is ergodic and such that for every $\ell>0$, $F(x)=x^{\ell}$ is not good for $(T_{1},T_{2})$, i.e. there exists subset $A\in\mathcal{X}$  such that $$\mu(A\cap T^{-n}_{1}A\cap T^{-n}_{2}A)<\mu(A)^{\ell}$$ for every $n\neq 0$.	
\end{thm}

\subsection{Notation and conventions}
A \emph{measure preserving system} (or a \emph{system} for short) is a tuple $(X,\mathcal{X},\mu,T_1$ $,\ldots,T_d)$, where $(X,\mathcal{X},\mu)$ is a probability space and $T_1,\ldots,T_d\colon X\to X$ are actions such that for all $A\in\mathcal{X}, 1\leq i\leq d$, $T^{-1}_{i}A\in\mathcal{X}$ and $\mu(T_{i}^{-1}A)=\mu(A)$. 
We use $\langle T_1,T_2,\ldots,T_d\rangle$ to denote the group spanned by the transformations $T_1,\ldots,T_d$.    
We say that $X$ is \emph{ergodic} for $\langle T_{1},\dots,T_{d}\rangle$ if $A\in\mathcal{X}, T_{i}^{-1}A=A$ for all $1\leq i\leq d$ implies that $\mu(A)=0$ or 1.

For a positive integer number $N$, the subset $\{1,\ldots,N\}$ is denoted by $[N]$.

\section{Combinatorial constructions} \label{Sec:Combinatorial}
In this section we study subsets of $\mathbb{N}^2$ and $\mathbb{N}^3$ satisfying special combinatorial conditions that help us construct the counter examples we need. The construction of such sets is inspired  by the methods used by Salem and Spencer \cite{SS} and Behrend  \cite{Beh} in building ``large'' subsets of $[N]$ with no arithmetic progressions of length 3.  The ways to make use of special subsets in Theorems \ref{Thm:2Recurrence}, \ref{Thm:3Recurrence} and \ref{Thm:Nilpotent} are motivated by the examples constructed in Bergelson, Host and Kra \cite{BHK} and Chu \cite{Chu}. 

 We remark that the combinatorial properties studied in this section are of independent interest.

\subsection{ Corner-free subsets of $\mathbb{N}^2$}
The first combinatorial problem we study is how large a subset $\L\subseteq [N]^2$ can be without a ``corner''.  
\begin{defn}
We say that a set $\L\subseteq [N]^{2}$ is \emph{corner-free} if $(x,y)$, $(x',y)$ and $(x,y')\in\L$ and $x-y=x'-y'$ implies that $x=x'$ and $y=y'$.
\end{defn}	
We have

\begin{thm}\label{Thm:Lfreeset}
	Let $\nu(N)$ denote the largest cardinality of corner-free subsets of $[N]^{2}$. Then $$\nu(N)>N^{2-\frac{4\log 2+\e}{\log\log N}}$$ as $N\to\infty$ for all $\e>0$.
\end{thm}

It is worth noting that Atjai and Szemer\'edi \cite{AS} had 
a similar estimate for the largest cardinality of the set $\L\subseteq [N]^{2}$ such that $(x,y)$, $(x',y)$ and $(x,y')\in\L, x'>x$ and $x-y=x'-y'$ implies that $x=x'$ and $y=y'$ (with an additional but not essential assumption that $x'>x$). Since our proof is different from their method, we write it down for completeness. We thank T. Ziegler for bringing us to this reference.

\begin{proof}

	Let $1\ll d\ll n$ be two parameters to be chosen later and assume that $n$ is divisible by $d^{2}$. Let $\L$ be the set of points $(x,y)\in [(2d)^{n}]^{2}$ such that the following condition holds: expand $x=x_{0}+x_{1}(2d-1)+\dots+x_{n-1}(2d-1)^{n-1}$ and $y=y_0+y_{1}(2d-1)+\dots+y_{n-1}(2d-1)^{n-1}$ $0\leq x_i,y_i\leq 2d-2$ for $0\leq i\leq n-1$. Consider the $n$ pairs of integers $(x_{0},y_{0}),(x_{1},y_{1}),\dots,(x_{n-1},y_{n-1})$.  Define $\L$ by $(x,y)\in \L$ if and only if among the pairs $(x_{0},y_{0}),(x_{1},y_{1}),\dots,(x_{n-1},y_{n-1})$, there are exactly $\frac{n}{d^{2}}$ of them that are equal to $(i,j)$ for all $0\leq i, j\leq d-1$. Recall for a set $S$ of cardinality $k$ and $k'$ that divides $k$, $\frac{k!}{(k/k')!^{k'}}$ is the number of different patitions of $S$ where each atom has exactly $k'$ elements.  Using this formula, we get that \[ \vert \Lambda \vert = \frac{n!}{((\frac{n}{d^{2}})!)^{d^{2}}}. \] 
	
	We claim that $\L$ satisfies the properties we are looking for. We first estimate the size of $\L$. 
	Set $n=d^{2}\omega(d)$, where $\o\colon\N\to\N$ is an increasing function such that $\o(d+1)-\o(d)=O(1)$, $\frac{\o(d)}{\log d}\to\infty$ and $\frac{\log\o(d)}{\log d}\to 0$ as $d\to\infty$. For every $N\in\N$, pick $d\in\N$ such that
	\begin{equation}\label{Beh4}
	\begin{split}
	(2d-1)^{d^{2}\o(d)}\leq N<(2d+1)^{(d+1)^{2}\o(d+1)}.
	\end{split}
	\end{equation}
By the Stirling formula, if $d$ and $n/d^{2}=\o(d)$ are large enough, we have that
	\begin{equation}\nonumber
	\begin{split}
	\quad\frac{n!}{((\frac{n}{d^{2}})!)^{d^{2}}} > \frac{n^{n}\sqrt{2\pi n}e^{-n}}{[(n/d^{2})^{n/d^{2}}\sqrt{2\pi(n/d^{2})}e^{-n/d^{2}}]^{d^{2}}}\frac{1}{C^{d^{2}}}
	\geq\frac{d^{2n}}{(\frac{\gamma n}{d^{2}})^{d^{2}/2}}=\frac{d^{2d^{2}\o(d)}}{(\gamma\o(d))^{d^{2}/2}}, 
	\end{split}
	\end{equation}
	where $\gamma=2\pi C^{2}$, and $C$ is a constant as close to 1 as we want. So 
	\begin{equation}\label{Beh5}
	\begin{split}
	&\quad\log \left(  \frac{N^{2}}{\vert L \vert}\right)<\log\left (\frac{(2d+1)^{2(d+1)^{2}\o(d+1)}}{\vert L \vert}\right) <\log\left (\frac{(2d+1)^{2(d+1)^{2}\o(d+1)}(\gamma\o(d))^{d^{2}/2}}{d^{2d^{2}\o(d)}} \right)
	\\&=2(d+1)^{2}\o(d+1)\log(2d+1)-2d^{2}\o(d)\log d+\frac{d^{2}}{2}(\log\gamma+\log\o(d))
	\\&=d^{2}\o(d)(2\log 2+o(1)),
	\end{split}
	\end{equation}
	where in the last step we repeatedly used the properties of $\o(d)$. On the other hand, by (\ref{Beh4}), we have
	$$\log N\geq d^{2}\o(d)\log(2d-1)$$
	and
	$$\log\log N< 2\log(d+1)+\log\o(d+1)+\log\log (2d+1),$$
	which implies that 
	\begin{equation}\label{Beh6}
	\begin{split}
	\frac{\log N}{\log\log N}>d^{2}\o(d)(1/2+o(1)).
	\end{split}
	\end{equation}	
	Combining (\ref{Beh5}) and (\ref{Beh6}), we have that 
	$$ \vert \L \vert > N^{2-\frac{4\log 2+\e}{\log\log N}}$$
	as $N\to\infty$ for all $\e>0$.
	
	\
	
		Now we show that $\Lambda$ is corner-free. Suppose that $(x,y)$, $(x',y)$ and $(x,y')$ belong to  $\L$  and that $x-y=x'-y'$. Expand $w=w_{0}+w_{1}(2d-1)+\dots+w_{n-1}(2d-1)^{n-1}, 0\leq w_{0},\dots,w_{n-1}\leq d-1$ for $w=x,y,x',y'$. Since $0\leq x_i,y_i,x_i',y_i'\leq d-1$ for every $0\leq i\leq n-1$, we have that necessarily
		\begin{equation}\label{equ1}
		\begin{split}
		y_{i}-x_{i}=y'_{i}-x'_{i}
		\end{split}
		\end{equation} 
		for all $0\leq i\leq n-1$. 
		
		If $y_{i}-x_{i}=-(d-1)$ for some $0\leq i\leq n-1$, then by the construction of $\L$, we have $x_{i}=d-1$ and $y_{i}=0$. By (\ref{equ1}),
		$y'_{i}-x'_{i}=y_{i}-x_{i}=-(d-1)$, and so  $x'_{i}=d-1$ and $y'_{i}=0$. Therefore $x_{i}=x'_{i}$ and $y_{i}=y'_{i}$.
		
		Now suppose that $-(d-1)\leq y_{i}-x_{i}\leq h$ implies that $(x_{i},y_{i})=(x'_{i},y'_{i})$ for all $0\leq i\leq n-1$ for some $-(d-1)\leq h\leq d-2$. We show that $y_{i}-x_{i}=h+1$ also implies that $(x_{i},y_{i})=(x'_{i},y'_{i})$ for all $0\leq i\leq n-1$. By the construction of $\L$, the number of the pairs $(x_{i},y_{i})$ such that $0\leq y_{i}-x_{i}\leq h$ is the same as the number of the pairs $(x_{i},y'_{i})$ such that $0\leq y'_{i}-x_{i}\leq h$, and is the same as the number of the pairs $(x'_{i},y_{i})$ such that $0\leq y_{i}-x'_{i}\leq h$. Therefore, by induction hypothesis, if $y_{i}-x_{i}=h+1$, we have $y'_{i}-x_{i}\geq h+1$ and $y_{i}-x'_{i}\geq h+1$.
		So 
		\begin{equation}\nonumber
		\begin{split}
		&\quad y'_{i}-x'_{i}=(y'_{i}-x_{i})+(y_{i}-x'_{i})-(y_{i}-x_{i})
		\\&\geq (h+1)+(h+1)-(h+1)=h+1=y_{i}-x_{i}.
		\end{split}
		\end{equation}
		By (\ref{equ1}), we have that $y'_{i}-x_{i}=y_{i}-x'_{i}=h+1=y_{i}-x_{i}$, which implies that $x_{i}=x'_{i}$ and $y_{i}=y'_{i}$ and we are done. 
		
		We conclude that \[ \nu(N)\geq \vert \L \vert  > N^{2-\frac{4\log 2+\e}{\log\log N}}\]
		as $N\to\infty$ for all $\e>0$.
\end{proof}

 \subsection{Three point free subsets of $\mathbb{N}^3$} 
 We study another combinatorial problem in this section.
 \begin{defn}
 Let $\L$ be a subset of $[N]^3$. We say that $\L$ is \emph{three point free} if $(x,y,z'),(x,y',z),(x',y,z)\in\L$ implies that $x=x', y=y', z=z'$.
 \end{defn}
 In particular, $(x,y,z')$ and $(x,y,z) \in \L$ implies that $z=z'$. so $\Lambda$ contains at most one point on each line parallel to $Z$-axis. Similarly, $\Lambda$ contains at most one point along any line parallel to the $X$ or $Y$-axis. 
 \begin{rem} To a three point free set $\Lambda\subseteq [N]^3$ we can associate a $N\times N$ matrix $A(\Lambda)=(a_{i,j})_{i,j\in [N]}$. To do so, we look at the line $\{(i,j,k): k\in [N]\}$. If there is a point of $\Lambda$ in such a line we set $a_{i,j}$ to be the unique integer in $[N]$ such that  $(i,j,a_{i,j})\in \Lambda$. If there is no point in such a line we just set $a_{i,j}=0$.  If $\Lambda$ is a three point free set, the matrix $A(\Lambda)$ has the following properties:
 \begin{enumerate}
 	\item For every $k\in [N]$, $k$ appears at most once in each row and each column of $A(\Lambda)$. \label{item:Matrix1}
 	\item For every $k\in [N]$, if $a_{i,j}=k$ and $a_{i',j'}=k$ then $a_{i,j'}=a_{i',j}=0$.  \label{item:Matrix2}
 \end{enumerate}
 
 Conversely, if $A$ is a $N\times N$ matrix that satisfies conditions (\ref*{item:Matrix1}) and (\ref*{item:Matrix2}), then there is a three point free set $\Lambda\subseteq [N]^3$ such that $A=A(\Lambda)$. The set $\Lambda$ is just $\{(i,j,a_{i,j}): a_{i,j}\neq 0\}$ and note that the cardinality of $\Lambda$ is the number of non-zero entries of $A(\Lambda)$.   
\end{rem}
 
 The combinatorial problem we study is how large such a set three point free set can be.
 It is clear that if $\Lambda \subseteq [N]^3$ is a three point free set, then $\vert\Lambda\vert \leq N^2$. We show that in fact $\vert\Lambda\vert$ can be sufficiently ``close'' to $N^{2}$:

\begin{thm} \label{Thm:3pointSet}
	Let $w(N)$ denote the largest cardinality of a three point free subset of $[N]^{3}$. 
Then
	$$w(N)>N^{2-\frac{4\log 2+2\e}{\log\log N}}$$
	as $N\to\infty$ for all $\e>0$.
\end{thm}  

\begin{proof}
Let $n$, $d$, $N$ and $\L$ be given in the proof of Theorem \ref{Thm:Lfreeset}. Define 
\[ V:=\Bigl\{ (x,y,z)\colon (x,y)\in \L \Bigr\}\subseteq [N]^3. \]
and
	\[ V_{s}:=\Bigl\{ (x,y,z)\in V\colon x+y+z=s  \Bigr\}. \]
We have that for big enough $N$, 
\[\sum_{s=0}^{3N-3}\vert V_{s}\vert=\vert V\vert = N\cdot N^{2-\frac{4\log 2+\e}{\log\log N}}.
\]
So there exists $0\leq s\leq 3N-3$ such that 
$$\vert V_{s}\vert\geq\frac{\vert V\vert}{3N}=N^{2-\frac{4\log 2+\e}{\log\log N}}/3>N^{2-\frac{4\log 2+2\e}{\log\log N}}$$
provided $N$ is large enough.
We verify that $V_{s}$ is three point free. Suppose that $(x,y,z')$, $(x,y',z)$ and $(x',y,z)$ belong to  $V_{s}$.
Then in particular we have that $(x,y)$, $(x,y')$ and $(x',y)$ belong to $\L$ and  $s-x-y'=z=s-x'-y$. So $x'-x=y'-y$.  Since $\L$ is corner-free, we have that $x=x'$ and $y=y'$. This implies that $z'=s-x-y=s-x'-y=z$ and we conclude that $V_{s}$ is three point free.

It follows immediately that \[ w(N)\geq \vert V_{s} \vert \geq  N^{2-\frac{4\log 2+2\e}{\log\log N}}\] 
 provided that $N$ is large enough. 
\end{proof}

\section{Nilsystems and affine nilsystems}
In all that follows, we make use of the class of nilsystems, specially of {\em affine nilsystems} and we briefly introduce them.
\subsection{Affine nilsystems with a single transformation}
Let $G$ be a group. For $a,b \in G$, $[a,b]:=aba^{-1}b^{-1}$ denotes the \emph{commutator} of $a$ and $b$ and for $A$ and $B$ subsets of $G$, and $[A,B]$ denotes the group generated by all the commutators $[a,b]$ for  $a\in A$ and $b\in B$.  The commutator subgroups are defined recursively as $G_1=G$ and $G_{j+1}=[G_j,G]$, $j\geq 1$. We say that $G$ is \emph{$d$-step nilpotent} if $G_{d+1}=\{ 1 \}$.  

Let $G$ a $d$-step nilpotent Lie group and $\Gamma$ be a discrete a cocompact subgroup of $G$. The compact manifold $G/\Gamma$ is a \emph{$d$-step nilmanifold}. The group $G$ acts on $G/\Gamma$ as left translations and there is a unique probability measure $\mu$ which invariant under such action (called the \emph{Haar measure}).  A dynamical system of the form $(G/\Gamma,\mathcal{B}(G/\Gamma),\mu,T_1,\ldots,T_n)$ is called a \emph{nilsystem}, where $\mathcal{B}(G/\Gamma)$ is the \emph{Borel} $\sigma$-algebra of $G/\Gamma$, and each $T_i$ is  given by the rotation by a fixed $\tau_i \in G$, i.e. $T_i\colon G/\Gamma \to G/\Gamma, x\mapsto \tau_i x$.  

An important class of such systems are the affine nilsystems. The affine nilsytems for a single transformation are defined as follows. Let ${\alpha}\in \mathbb{T}^d$ and let $A$ be a $d\times d$ \emph{unipotent} integer matrix (i.e. $(A-I)^d=0$). Let $T\colon \mathbb{T}^d\to \mathbb{T}^d$ be the affine transformation $x\mapsto Ax+{\alpha}$. Let $G$ be the group of transformations of $\mathbb{T}^d$ generated by $A$ and the translations of $\mathbb{T}^d$. That is,  every element $g\in G$ is a map $x\mapsto A^i x +\beta$ for some $i\in \mathbb{Z}$ and $\beta \in \mathbb{T}^d$. The group $G$ acts transitively on $\mathbb{T}^d$ and we may identify this space with $G/\Gamma$, where $\Gamma$ is  the stabilizer of 0, which consists of the powers of $A$. The  system $(\mathbb{T}^d,\mu^{\otimes d},T)$ is called an \emph{affine nilsystem} (here $\mu$ is the Haar measure on $\mathbb{T}$). Properties such as transitivity, minimality, ergodicity and unique ergodicity are equivalent for a system in this class and this can be checked by looking at the rotation induced by ${\alpha}$ on the projection $\mathbb{T}^d/Ker(A-I)$ \cite{P}.

\subsection{Affine nilsystems with several transformations} When we consider different affine transformations $T_i\colon \mathbb{T}^d\to \mathbb{T}^d$, $x\mapsto A_i x + \alpha_i$, where $A_i$ is unipotent for every $i=1,\ldots,n$, we can still regard the system $(\mathbb{T}^d,\mathcal{B}(\mathbb{T}^d),\mu^{\otimes d}$ , $T_1,\ldots, T_n)$ as a nilsystem as long as the matrices commute. Indeed, let $G$ be the group of transformations of $\mathbb{T}^d$ generated by the matrices $A_1,\ldots,A_n$ and the translations of $\mathbb{T}^d$. Then every element $g\in G$ is a map $x\mapsto A(g)x+ \beta(g)$, where $A(g)=A_1^{m_1}\cdots A_n^{m_n}$, $m_1,\ldots,m_n \in \mathbb{Z}$ and  $\beta(g)\in \mathbb{T}^d$.  

A simple computation shows that if $g_1,g_2\in G$, the commutator $[g_1,g_2]$ is the map $x\mapsto x + (A(g_1)-I)\beta(g_2) + (A(g_2)-I)\beta(g_1)$ and thus is a translation of $\mathbb{T}^d$. In the other hand, if $g\in G$ and  $\beta \in \mathbb{T}^d$, then $[g,\beta]$ is the translation $x\mapsto x+(A(g)-I)\beta$. It follows that the iterated commutator $[\cdots[[g_1,g_2],g_3]\cdots g_k]$ belongs to $\mathbb{{T}}^d$ and is contained in the image of  $(A(g_3)-I)\cdots(A(g_k)-I)$. If $k$ is large enough, this product is trivial. So $G$ is a nilpotent Lie group. The torus $\mathbb{T}^d$ can be identified with $G/\Gamma$, where $\Gamma$ is the stabilizer of 0, which is the group generated by the matrices $A_1,\ldots, A_n$. We refer to 
$(\mathbb{T}^d,\mathcal{B}(\mathbb{T}^d),\mu^{d},T_1,\ldots,T_n)$ as an affine nilsystem with $n$ transformations. It is worth noting that the transformations $T_i$ and $T_j$ commute if $(A_i-I)\alpha_j= (A_{j}-I)\alpha_i$ in $\mathbb{T}^d$. 

By a theorem from Leibman \cite{L1}, we get:
\begin{prop}\label{Property:MinimalErgodic}
Let $(\mathbb{T}^d,\mathcal{B}(\mathbb{T}^d),\mu^{d},T_1,\ldots,T_n)$ be an affine nilsystem with $n$ transformations. Then the properties  of transitivity, minimality, ergodicity and unique ergodicity under the action of $\langle T_1\ldots,T_n \rangle$ are equivalent. 
\end{prop}

Recall that $(\mathbb{T}^d,\mathcal{B}(\mathbb{T}^d),\mu^{d},T_1,\ldots,T_n)$ is \emph{minimal} under $\langle T_1\ldots,T_n \rangle$ if the orbit closure of every $x\in\mathbb{T}^d$ under $\langle T_1\ldots,T_n \rangle$ is $\mathbb{T}^d$ (in this paper we do not need the concepts of transitivity and unique ergodicity).
We refer to \cite{L1} for a modern reference on nilmanifolds and nilsystems.

\section{Proofs of the main theorems}

We are now ready to prove the main theorems.

\begin{proof}[Proof of Theorem \ref{Thm:2Recurrence}]
	Let $\alpha,\beta \in\mathbb{R}\backslash\mathbb{Q}$ be rationally independent numbers. Let $X=\mathbb{T}^{6}$ with transformations
	$$T_{1}(x_{1},x_{2},x_{3};y_{1},y_{2},y_{3})=(x_{1}+\a,x_{2}+\b,x_{3};y_{1}+x_{1},y_{2},y_{3}+x_{1}+x_{2}+x_{3})$$
	and $$T_{2}(x_{1},x_{2},x_{3};y_{1},y_{2},y_{3})=(x_{1},x_{2}+\b,x_{3}+\a;y_{1},y_{2}+x_{3},y_{3}+x_{1}+x_{2}+x_{3}).$$
	
	We have that $(X,\mathcal{B}(\mathbb{T}^6),\mu\otimes\mu\otimes\mu\otimes\mu\otimes\mu\otimes\mu,T_{1},T_{2})$ is an affine nilsystem with two transformations, where $\mu$ is the Haar measure on $\mathbb{T}$. Notice that
	\begin{equation}\nonumber
	\begin{split}
	&\quad T_{1}T_{2}(x_{1},x_{2},x_{3};y_{1},y_{2},y_{3})=T_{2}T_{1}(x_{1},x_{2},x_{3};y_{1},y_{2},y_{3})
	\\&=(x_{1}+\a,x_{2}+2\b,x_{3}+\a;y_{1}+x_{1},y_{2}+x_{3},y_{3}+x_{1}+x_{2}+x_{3}+\a+\b).
	\end{split}
	\end{equation}
	
	We first claim that the system is minimal and ergodic. To see this, let $(x_1,x_2,x_3;y_1,y_2,y_3)$ and $(x_1',x_2',x_3';y_1',y_2',y_3')\in \mathbb{T}^6$. It is not hard to see that $(x_1',x_2';x_3,y_1',y_2,y_3')$ belongs to the orbit closure of $(x_1,x_2,x_3;y_1,y_2,y_3)$ under the transformation $T_1$. We also have that $(x_1',x_2',x_3';y_1',y_2',y_3')$ is in the orbit closure of    $(x_1',x_2',x_3,y_1',y_2,y_3')$ under the transformation $T_2$ (here we use the fact that for fixed $x_1$, the transformation $(x_2,x_3,y_2,y_3)\mapsto (x_2+\beta,x_3+\alpha,y_2+x_3,y_3+x_1+x_2+x_3)$ is minimal in $\mathbb{T}^4$). We conclude that $(x_1',x_2',x_3',y_1',y_2',y_3')$ is in the orbit closure of $(x_1,x_2,x_3;y_1,y_2,y_3)$ under $\langle T_1,T_2\rangle$. Since since the points are arbitrary, the system is minimal and hence ergodic by Proposition \ref{Property:MinimalErgodic}. 

\
	
	Let $N\in\N$ to be chosen later and $\L\subseteq [N]^{3}$ be a three point free set. For $(a,b,c)\in[N]^{3}$, denote $$Q_{a,b,c}=\Bigl(\frac{a}{N},\frac{a}{N}+\frac{1}{2N}\Bigr)\times\Bigl(\frac{b}{N},\frac{b}{N}+\frac{1}{2N}\Bigr)\times\Bigl(\frac{c}{N},\frac{c}{N}+\frac{1}{2N}\Bigr)\subseteq\mathbb{T}^{3}$$ and let $B=\cup_{(a,b,c)\in\L}Q_{a,b,c}$, $A=\mathbb{T}^{3}\times B$. Then 
	$$\mu\otimes\mu\otimes\mu\otimes\mu\otimes\mu\otimes\mu(A)=\frac{\vert\L\vert}{8N^{3}}.$$
	On the other hand,
	\begin{equation}\nonumber
	\begin{split}
	&\quad \mu\otimes\mu\otimes\mu\otimes\mu\otimes\mu\otimes\mu(A\cap T^{-n}_{1}A\cap T^{-n}_{2}A)
	\\&=\int_{\mathbb{T}^{6}}\bold{1}_{B}(y_{1},y_{2},y_{3})\bold{1}_{B}(y'_{1},y_{2},y'_{3})\bold{1}_{B}(y_{1},y'_{2},y'_{3})d\mu\otimes\dots\otimes\mu(x_{1},x_{2},x_{3};y_{1},y_{2},y_{3}),
	\end{split}
	\end{equation}
	where
	$$y'_{1}=y_{1}+nx_{1}+\binom{n}{2}\a, ~~~  y'_{2}=y_{2}+nx_{2}+\binom{n}{2}\a, ~~~ y'_{3}=y_{3}+n(x_{1}+x_{2}+x_{3})+\binom{n}{2}(\a+\b).$$
	
	Suppose that the product of the functions inside the integral is nonzero. Then we may assume that 
	$(y_{1},y_{2},y_{3})\in Q_{a,b,c}$, $(y'_{1},y_{2},y_{3}')\in Q_{a',b,c'}$ and $(y_{1},y'_{2},y'_{3})\in Q_{a,b',c'}$ for some $(a,b,c)$, $(a',b,c')$ and $(a,b',c')$ that belong to $\L$. Since $\L$ is three point free, we have that $a=a'$, $b=b'$ and $c=c'$. This implies that $$nx_{1}+\binom{n}{2}\a,~~ nx_{2}+\binom{n}{2}\a,~~ y_{3}+n(x_{1}+x_{2}+x_{3})+\binom{n}{2}(\a+\b)\in(-\frac{1}{2N},\frac{1}{2N}).$$ Therefore,   
	\begin{equation}\nonumber
	\begin{split}
	&\quad \mu\otimes\mu\otimes\mu\otimes\mu\otimes\mu\otimes\mu(A\cap T^{-n}_{1}A\cap T^{-n}_{2}A)
	\\&=\int_{\mathbb{T}^{6}}\bold{1}_{B}(y_{1},y_{2},y_{3})\bold{1}_{B}(y'_{1},y_{2},y'_{3})\bold{1}_{B}(y_{1},y'_{2},y'_{3})d\mu\otimes\dots\otimes\mu(x_{1},x_{2},x_{3};y_{1},y_{2},y_{3})
	\\&= \frac{1}{N^{3}}\int_{\mathbb{T}^{3}}\bold{1}_{B}(y_{1},y_{2},y_{3})d\mu(y_{1})d\mu(y_{2})d\mu(y_{3})=\frac{1}{N^{3}}\cdot\frac{\vert\L\vert}{8N^{3}}.
	\end{split}
	\end{equation}
	
	We have that $\mu\otimes\mu\otimes\mu\otimes\mu\otimes\mu\otimes\mu(A\cap T^{-n}_{1}A\cap T^{-n}_{2}A)=\frac{\vert\L\vert}{8N^{6}}\leq \frac{\vert\L\vert^\ell}{(8N^{3})^{\ell}}=\mu(A)^{\ell}$
	if and only if 
	\begin{equation} \label{Eq:Exponent2}
	\ell\leq 1+\frac{3}{3+\frac{\log(8)}{\log(N)}-\frac{\log(\vert \L \vert)}{\log(N)}}.
	\end{equation} 
	By Theorem \ref{Thm:3pointSet}, we can take $\Lambda$ of cardinality larger than $N^{2-\epsilon}$ and thus the right hand side in \eqref{Eq:Exponent2} can be as close to 4 as we want. This finishes the proof. 
\end{proof}

\begin{proof}[Proof of Theorem  \ref{Thm:Nilpotent}]
	Let $\alpha \in\mathbb{R}\backslash\mathbb{Q}$. Let $X=\mathbb{T}^{3}$ with transformations
	$$T_{1}(x,y,z)=(x+\alpha,y+x,z)$$
	and $$T_{2}(x,y,z)=(x+\alpha,y,z+x).$$
	
	It is an affine nilsystem with two transformations. We first claim that the system $(X,\mathcal{B}(\mathbb{T}^3),\mu\otimes\mu\otimes\mu,T_{1},T_{2})$ is minimal and ergodic, where $\mu$ is the Haar measure on $\mathbb{T}$. To see this, take $(x_0,y_0,z_0)\in \mathbb{T}^{3}$ and note that the closure of the orbit of this point under the transformation $T_1$ is $\mathbb{T}\times \mathbb{T} \times \{z_0\}$. Let $(x,y,z)\in \mathbb{T}^{3}$ be an arbitrary point and notice that this point is contained in the  orbit closure of $(x_0,y,z_0)$ under the transformation $T_2$. So the closure of the orbit of $(x_0,y_0,z_0)$ under $\langle T_1,T_2\rangle$ is  $\mathbb{T}^3$. We conclude that the system is minimal and thus ergodic by Proposition \ref{Property:MinimalErgodic}. 
	
	It is easy to verify that $[T_{1},T_{2}](x,y,z)=(x,y+\a,z-\a)$ commutes with $T_{1}$ and $T_{2}$. So $T_{1}$ and $T_{2}$ generate a 2-step nilpotent group.
	
	Let $N\in\N$ to be chosen later and $\L\subseteq [N]^{2}$ be a corner-free set. For $(a,b)\in[N]^{2}$, denote $$Q_{a,b}=\Bigl(\frac{a}{N},\frac{a}{N}+\frac{1}{2N}\Bigr)\times\Bigl(\frac{b}{N},\frac{b}{N}+\frac{1}{2N}\Bigr)\subseteq\mathbb{T}^{2}$$ and let $B=\cup_{(a,b)\in\L}Q_{a,b}$ and $A=\mathbb{T}\times B$. Then 
	$$\mu\otimes\mu\otimes\mu(A)=\frac{\vert\L\vert}{4N^{2}}.$$
	On the other hand,
	\begin{equation}\nonumber
	\begin{split}
	&\quad \mu\otimes\mu\otimes\mu(A\cap T^{-n}_{1}A\cap T^{-n}_{2}A)
	\\&=\int_{\mathbb{T}^{3}}\bold{1}_{B}(y,z)\bold{1}_{B}\Bigl(y+nx+\binom{n}{2}\a,z\Bigr)\bold{1}_{B}\Bigl(y,z+nx+\binom{n}{2}\a\Bigr)d\mu(x)d\mu(y)d\mu(z).
	\end{split}
	\end{equation}
	Suppose that the product of the functions inside the integral is nonzero. Write $c=nx+\binom{n}{2}\a$ for convenience. Then we may assume that 
	$(y,z)\in Q_{a,b}$, $(y+c,z)\in Q_{a',b}$ and $(y,z+c)\in Q_{a,b'}$ for some $(a,b)$, $(a',b)$   $(a,b')$ that belong to $\L$. Since $(y+c)-y=(z+c)-z$, by the construction of the $Q_{a,b}$'s, we must have that $a'-a=b'-b$. Since $\L$ is $L$-free, we have that $a=a'$ and $b=b'$. This implies that $c=nx+\binom{n}{2}\a\in(-\frac{1}{2N},\frac{1}{2N})$. Therefore,   
	\begin{equation}\nonumber
	\begin{split}
	&\quad \mu\otimes\mu\otimes\mu(A\cap T^{-n}_{1}A\cap T^{-n}_{2}A)
	\\&=\int_{\mathbb{T}^{3}}\bold{1}_{B}(y,z)\bold{1}_{B}\Bigl(y+nx+\binom{n}{2}\a,z\Bigr)\bold{1}_{B}\Bigl(y,z+nx+\binom{n}{2}\a\Bigr)d\mu(x)d\mu(y)d\mu(z)
	\\&\leq \frac{1}{N}\int_{\mathbb{T}^{2}}\bold{1}_{B}(y,z)d\mu(y)d\mu(z)=\frac{1}{N}\cdot\frac{\vert\L\vert}{4N^{2}}.
	\end{split}
	\end{equation}
	
	We have that $\mu\otimes\mu\otimes\mu(A\cap T^{-n}_{1}A\cap T^{-n}_{2}A)=\frac{\vert\L\vert}{4N^{3}}\leq \frac{\vert\L\vert^\ell}{(4N^{2})^{\ell}}=\mu(A)^{\ell}$
	if and only if \begin{equation} \label{Eq:Exponent}
	\ell\leq 1+\frac{1}{2+\frac{\log(4)}{\log(N)}-\frac{\log(\vert \L \vert)}{\log(N)}}
	\end{equation}

	By Theorem \ref{Thm:Lfreeset}, we can take $\Lambda$ of cardinality larger than $N^{2-\epsilon}$ and thus the right hand side in \eqref{Eq:Exponent} can be arbitrarily large. This finishes the proof.    
\end{proof}

For Theorem \ref{Thm:3Recurrence},
we make use of the following theorem in \cite{BHK} which gives a negative answer to Question \ref{Q:Poincare} in the non-ergodic case. 

\begin{thm} \label{Thm:NonErgodic}
	Let $(\mathbb{T}^2,\mathcal{B}(\mathbb{T}^2),\mu\otimes \mu, T)$ be a measure preserving system on $\mathbb{T}^2$, where $\mu$ is the Lebesgue measure on $\mathbb{T}$ and $T$ is the transformation $(x,y)\mapsto (x,y+x)$.
	For every $\ell>0$, there exists $B\in\mathcal{B}(\mathbb{T}^2)$ such that $$\mu(B\cap T^{-n} B \cap T^{-2n} B) \leq \mu(B)^{\ell}$$ for every $n\neq 0$.  
\end{thm}
\begin{rem}
	 It is worth noting that the subset $B$ is constructed with the help of Behrend Theorem's on subsets of integers with no arithmetic progressions of length 3 \cite{Beh}.
\end{rem}
The system in Theorem \ref{Thm:NonErgodic} is not ergodic.
The idea to prove Theorem \ref{Thm:3Recurrence} is to take an extension of the system in Theorem \ref{Thm:NonErgodic} and add a transformation to make it ergodic. 
\begin{proof}[Proof of Theorem \ref{Thm:3Recurrence}]
	Let $\mu$ be the Haar measure on $\mathbb{T}$ and let $\alpha \in \mathbb{R}\setminus \mathbb{Q}$. Consider $X=\mathbb{T}\times	\mathbb{T} \times \mathbb{T}$ and its Haar measure $\mu\otimes \mu \otimes\mu$. Let $$T_1(x,y,z)=(x+\alpha, y, z+y),$$ $T_2=T_1^2$ and $$T_3(x,y,z)=(x,y+\alpha,z+x).$$ Since $$T_1T_3(x,y,z)=T_3T_1(x,y,z)=(x+\alpha,y+\alpha,z+x+y+\alpha),$$ $T_1$, $T_2$ and $T_3$ commute. We leave its proof to the readers that the action generated by $\langle T_1,T_2,T_3\rangle$ is minimal and hence ergodic in $X=\mathbb{T}\times	\mathbb{T} \times \mathbb{T}$ by Proposition \ref{Property:MinimalErgodic}. 
	By Theorem \ref{Thm:NonErgodic}, for every $\ell>0$, we can find a subset $B\subseteq \mathbb{T}\times	\mathbb{T}$ such that $$\mu \otimes \mu(B\cap T^{-n}B \cap T^{-2n}B) \leq \mu(B)^{\ell}$$ for all $n\neq 0$. Now it suffices to consider the set $A=\mathbb{T}\times B \subseteq\mathbb{T}\times \mathbb{T}\times \mathbb{T} $   and notice that
	\begin{equation}\nonumber
	\begin{split}
	   &\quad\mu\otimes \mu \otimes \mu(A\cap T_1^{-n} A \cap T_2^{-n} A \cap T_3^{-n} A)\\&\leq \mu\otimes \mu (B\cap T^{-n} B \cap T^{-2n} B) \leq \mu\otimes \mu (B)^{\ell}=\mu\otimes \mu\otimes \mu (A)^{\ell}
	\end{split}
	\end{equation}
 for every $n\neq 0$. 
\end{proof}

\end{document}